\documentclass[12pt]{article}
\usepackage{amsfonts,amssymb,amscd,amsmath,enumerate,verbatim,calc, times,url}
\usepackage{amsthm, psfrag,latexsym,epsfig,mdwlist,graphicx}

\theoremstyle{plain}
\newtheorem{theorem}{Theorem}
\newtheorem{lemma}[theorem]{Lemma}

\theoremstyle{definition}
\newtheorem{definition}[theorem]{Definition}

\newcommand{\vs}{v_1,\ldots,v_n}                

\newcommand{\D}{\Delta}                         
              
\newcommand{\st}{\ | \ }                        
\newcommand{\tuple}[1]{\langle #1 \rangle}      
\newcommand{\void}[1]{}

\newcommand{\erase}[1]{}

\newcommand{\lk}{{\rm{lk}}\ }                   

\newcommand{\cocoa}{\mbox{\rm C\kern-.13em o\kern-.07 em C\kern-.13em o\kern-.15em A}} 
\newcommand{\cocoax}{\mbox{C\kern-.13em o\kern-.07 em C\kern-.13em o\kern-.15em A}} 
\newcommand{\cocoal}{\mbox{\rm C\kern-.13em o\kern-.07 em C\kern-.13em o\kern-.15emA\kern-.1em L}}

\newcommand{\todo}[1]{\vspace{5 mm}\par \noindent
\marginpar{\textsc{ToDo}}
\framebox{\begin{minipage}[c]{0.95 \textwidth}
\tt #1 \end{minipage}}\vspace{5 mm}\par}

\renewcommand{\todo}[1]{}

\newcommand{\idiot}[1]{\vspace{5 mm}\par \noindent
\framebox{\begin{minipage}[c]{0.95 \textwidth}
\tt #1 \end{minipage}}\vspace{5 mm}\par}

\renewcommand{\idiot}[1]{}

\newcommand{\sm}{\setminus}

\date{}

\author{Sara Faridi\thanks{Department of
Mathematics and Statistics, Dalhousie University, Halifax, Canada, 
faridi@mathstat.dal.ca, +1(902)-494-2658. 
Research supported by NSERC.}}

\title{\Large \sc The Betti numbers of Stanley-Reisner ideals of
  simplicial trees}
\begin{document}

\maketitle
\begin{abstract} We provide a simple method to compute the Betti
  numbers of the Stanley-Reisner ideal of a simplicial tree and its
  Alexander dual.  
 \end{abstract}


\bigskip

 
Simplicial trees~\cite{F1} are a class of flag complexes initially
studied for the properties of their facet ideals. In this short note
we give a short and straightforward method to compute the Betti
numbers of their Stanley-Reisner ideals. 

The \emph{Betti numbers} of a homogeneous ideal $I$ in a polynomial
ring $R$ over a field are the ranks of the free modules appearing in a minimal free resolution 
$$0\rightarrow {\displaystyle
  \oplus_{d}}R(-d)^{\beta_{p,d}}\rightarrow\cdots{\displaystyle
  \rightarrow \oplus_{d}}R(-d)^{\beta_{0,d}}\rightarrow I \rightarrow
0$$ of $I$. Here $R(-d)$ denotes the graded free module obtained by
shifting the degrees of elements in $R$ by $d$. The numbers
$\beta_{i,d}$, which we shall refer to as the $i$th
$\mathbb{N}$-\emph{graded Betti numbers} of degree $d$ of $I$, are
independent of the choice of the graded minimal finite free resolution.

 \begin{definition}[simplicial complex] 
   A \emph{simplicial complex} $\D$ over a set of vertices $V=\{ \vs
   \}$ is a collection of subsets of $V$, with the property that $\{
   v_i \} \in \D$ for all $i$, and if $F \in \D$ then all subsets of
   $F$ are also in $\D$. An element of $\D$ is called a \emph{face} of
   $\D$.  The maximal faces of $\D$ under inclusion are called
   \emph{facets} of $\D$.  A \emph{subcollection} of $\D$ is a
   simplicial complex whose facets are also facets of $\D$; in other
   words a simplicial complex generated by a subset of the set of
   facets of $\D$. $A \subseteq V$, the \emph{induced subcomplex of $\D$ on
     $A$}, denoted by $\D_A$, is defined as $\D_A=\{F\in \D \st
   F\subseteq A \}.$
 \end{definition}

\begin{definition} Let $\D$  be a simplicial complex with vertex set
$x_1,\ldots,x_n$ and $R=k\left[x_1,\dots,x_n\right]$ be a polynomial
  ring over a field $k$.  The \emph{Stanley-Reisner  ideal} of $\D$
  is defined as $I_{\D}=( \prod_{x_i \in F}x_i \st F \notin \D).$
\end{definition}

\begin{definition}[\cite{F1} leaf, joint] A facet $F$ of a 
simplicial complex is called a \emph{leaf} if either $F$ is the only
facet of $\D$ or for some facet $G\neq F$ of $\D$ we have $F \cap H
\subseteq G$ for all other facets $H$ of $\D$.  Such a facet $G$ is
called a \emph{joint} of $F$.
\end{definition}

\begin{definition}[\cite{F1} tree, forest]\label{d:tree} A connected 
simplicial complex $\D$ is a \emph{tree} if every nonempty
subcollection of $\D$ has a leaf.  If $\D$ is not necessarily
connected, but every subcollection has a leaf, then $\D$ is called a
\emph{forest}.
\end{definition}

\begin{theorem}[\cite{F2}~Theorem~2.5]\label{t:induced} An induced 
  subcomplex of a simplicial tree is a simplicial forest.\end{theorem}

\begin{definition}[link] Let  $\D$ 
be a simplicial complex over a vertex set $V$ and let $F$ be a face of
$\D$. The \emph{link} of $F$ is defined as $\lk_\D(F)=\{G \in \D \st F
\cap G = \emptyset \ \& \ F \cup G \in \D \}.$
\end{definition}

\begin{lemma}[A link in a tree is a forest]\label{l:links}
 If $\D$ is a tree and $F$ is a face of $\D$, then $\lk_\D(F)$ is a
 forest.
\end{lemma}

   \begin{proof} Suppose $\lk_\D(F)=\tuple{G_1,\ldots,G_s}$ where 
    $G_i$ is a subset of a facet $F_i=F\cup G$ of $\D$. Now suppose
     $\Gamma=\tuple{G_{a_1},\ldots,G_{a_r}}$ is a subcollection of
     $\lk_\D(F)$. We need to show that $\Gamma$ has a leaf. Let
     $\tuple{F_{a_1},\ldots,F_{a_r}}$ be the corresponding
     subcollection of $\D$, which must have a leaf, say $F_{a_1}$ and
     a joint, say $F_{a_2}$. Then we have $F_{a_i} \cap
     F_{a_1}\subseteq F_{a_2}$ for $i=3,\ldots,r.$ But since
     $F_{a_i}=F\cup G_{a_i}$ and $F\cap G_{a_i}=\emptyset$ for all
     $i$, we must have $G_{a_i} \cap G_{a_1}\subseteq G_{a_2}$ for
     $i=3,\ldots,r$ which means that $G_{a_1}$ is a leaf of $\Gamma$.
   \end{proof}


We will combine the above two facts with Hochster's formula for Betti
numbers of the ideal and its dual~\cite{BCP}.

\begin{theorem}[\cite{BCP}]
  Let $k$ be a field and $\D$ a simplicial complex over vertex set
  $V$. Then
\begin{align}
\beta_{i,j}(I_\D)&=\displaystyle \sum_{A\subseteq V, \ |A|=j}\dim_k \widetilde{H}_{j-i-2}(\D_A;k)\label{e:betti}\\
\beta_{i,j}(I_\D^\vee)&=\displaystyle \sum_{A\subseteq V, \ |A|=j} \dim_k \widetilde{H}_{i-1}(\lk_\D(V \sm A;k)).\label{e:d-betti}
\end{align}
\end{theorem}

If $\D$ is a tree, the following theorem shows how to find Betti
numbers of $I_\D$, and along the way also gives a proof of the fact
that $I_\D$ has a linear resolution. This last statement is not
unknown, it follows also from Fr\"oberg's characterizations of edge
ideals with linear resolutions~\cite{Fr} along with observations
in~\cite{HHZ}, and is also proved in~\cite{CF}.

\begin{theorem} Let $\D$ be a simplicial tree with vertex set $V$. Then 
$\D$ is a flag complex, $I_\D$ has a linear resolution, and the Betti
  numbers of $I_\D$ can be computed by 
$$\beta_{i,j}(I_\D)=\left \{
\begin{array}{cl}
\displaystyle \sum_{A\subseteq V, \ |A|=j}(\mbox{number of connected components of } \D_A -1)& j=i+2
\\
0 & \mbox{otherwise.}
\end{array}
\right .
$$
\end{theorem}

  \begin{proof} By (\ref{e:betti}) we know that we are looking at the 
    reduced homology modules of $\D_A$ for various $A\subseteq V$. For
    a given $A$, we know that $\D_A$ is a forest, and every connected
    component is a tree and therefore
    acyclic~(\cite{F2}~Theorem~2.9). Therefore, for each such $A$ the
    only possible nonzero reduced homology is the $0$th one, that is
    when $|A|-i-2=0$ or $|A|=i+2$. The formula now just follows.
    
    In particular, $\beta_0$ is only positive in degree 2, which
    implies that $\D$ is a flag complex, and the fact that the
    resolution is linear is evident  from the way the Betti numbers grow.
  \end{proof}

\begin{theorem} Let $\D$ be a simplicial tree with vertex set $V$ of 
cardinality $n$. Then the $I_\D^\vee$ has projective dimension 1, and
its Betti numbers are
$$\beta_{i,j}(I_\D^\vee)=\left \{
\begin{array}{ll}
\hspace{.2in}\mbox{number of facets of $\D$ of cardinality } n-j & i=0\\
&\\
\displaystyle  \sum_{A\subseteq V, \ |A|=j} \hspace{-.2in}(\mbox{number of connected components of } \lk_\D(V\sm A) -1) & i=1\\
&\\
\hspace{.2in}0 & \mbox{otherwise.}
\end{array}
\right .
$$
\end{theorem}

  \begin{proof} This follows from (\ref{e:d-betti}). Note that in this case
   we are looking at the homology modules of $\lk_\D(V\sm A)$ for
   $A\subseteq V$. By Lemma~\ref{l:links} $\lk_\D(V\sm A)$ is a forest,
   and so since all the connected components are acyclic, we
   only have possible homology in degrees -1 (if the link is empty)
   and 0.  
   
   The case $i=1$ is the $0$th homology, and we are counting the
   numbers of connected components minus 1, which  is straightforward.

    In the case $i=0$, we are counting only those $A\subset V$ where
    $\lk_\D(V\sm A)=\{\emptyset\}$, or equivalently $V\sm A$ is
    a facet of $\D$. So the formula for the case $i=0$ follows. 
  \end{proof}



\begin{thebibliography}{1234}
  
\bibitem[BCP]{BCP} D. Bayer, H. Charalambous, and S. Popescu, \emph{Extremal
  Betti numbers and Applications to Monomial Ideals}, J. Alg. 221
  (1999), 497--512.

\bibitem[CF]{CF} E. Connon, S. Faridi, \emph{Chorded complexes and a
  necessary condition for a monomial ideal to have a linear resolution,}
  arXiv:1209.5089.

  \bibitem[F1]{F1} S.~Faridi, \emph{The facet ideal of a simplicial
  complex,} Manuscripta Mathematica 109, 159--174 (2002).

 \bibitem[F2]{F2} S.~Faridi, \emph{Monomial resolutions supported by
    simplicial trees,} arXiv:1202.0750.
  
  \bibitem[Fr]{Fr} R. Fr\"oberg, \emph{On Stanley-Reisner rings,}
    Topics in Algebra, Part 2 (Warsaw, 1988) Banach Center Publ.,
    vol. 26, PWN, Warsaw, 1990, pp. 57--70.


\bibitem[HHZ]{HHZ} J. Herzog, T. Hibi, X. Zheng, \emph{Dirac’s theorem
  on chordal graphs and Alexander duality,} European J. Combin. 25
  (2004), 949–960.
\end{thebibliography}
\end{document}